\documentclass[12pt]{amsart}
\usepackage{enumitem, color, amssymb}
\usepackage{tikz}
\usepackage{verbatim}
\usepackage{hyperref}
\usepackage{cleveref}
\usepackage{enumerate}
\usepackage{geometry}
\usepackage{fullpage}
\usepackage{multirow}
\usepackage{array}
\usepackage{bookmark}
\usepackage{longtable}
\usepackage{caption}
\usepackage{xcolor}
\usepackage{float}
\usepackage{subcaption}
\usepackage{graphicx}
\newtheorem{theorem}{Theorem}[section]
\newtheorem{lemma}[theorem]{Lemma}
\newtheorem{corollary}[theorem]{Corollary}
\newtheorem{proposition}[theorem]{Proposition}

\newtheorem{conjecture}[theorem]{Conjecture}

\theoremstyle{definition}

\theoremstyle{remark}

\numberwithin{equation}{section}

\def \l {\lambda}

\def\Tr{{\rm Tr}}

\def\F{{\mathbb F}}
  \def\Fq{{\F_{q}}}
\def\Z{{\mathbb Z}}

\def\Q{{\mathbb Q}}

\newcommand*\mathinhead[2]{\texorpdfstring{$\boldsymbol{#1}$}{#2}}

\def\Im{\mathrm{Im}\,}
\def\SL{\mathrm{SL}}

\def\N{\mathbb N}

\def\({\left(}
\def\){\right)}
\def \l {\lambda}

\def\p{\varphi}

\def\Z{\mathbb{Z}}
\def\Q{\mathbb{Q}}
\def\F{\mathbb{F}}

\def \eps{\varepsilon}
\newcommand{\fp}
{\mathbb{F}_p}

\newcommand{\fq}
{\mathbb{F}_q}

\newcommand{\fphat}
{\widehat{\mathbb{F}_{p}^{\times}}}

\newcommand{\fqhat}
{\widehat{\mathbb{F}_{q}^{\times}}}

\newcommand*\HYPERskip{&}
\catcode`,\active
\newcommand*\pFq{
\begingroup
\catcode`\,\active
\def ,{\HYPERskip}%
\doHyper
}
\catcode`\,12
\def\doHyper#1#2#3#4#5{%
\, _{#1}F_{#2}\left[\begin{matrix}#3 \smallskip \\  #4\end{matrix} \; ; \; #5\right]%
\endgroup
}

\catcode`,\active

\catcode`\,12

\catcode`,\active

\catcode`\,12

\catcode`,\active

\catcode`\,12

\begin{document}

\title{Hypergeometric Moments and Hecke Trace Formulas}

\author{Brian Grove}
\address{Mathematics Department\\
  Louisiana State University\\
  Baton Rouge, Louisiana}
\email{briangrove30@gmail.com}

\date{}

\begin{abstract}
Moments for hypergeometric functions over finite fields were studied in the work of Ono, Pujahari, Saad, and Saikia for several $_{2}F_{1}$ and $_{3}F_{2}$ cases. We generalize their work to prove results for new cases where the hypergeometric data is defined over $\Q$ and primitive. These new moments are established using Hecke trace formulas of hypergeometric origin recently established by Hoffman, Li, Long, and Tu. We also obtain several algebraic formulas in the finite field setting and present conjectures for additional $_{2}F_{1}$ and $_{3}F_{2}$ moments.
\end{abstract}

\maketitle

\section{Introduction}

\subsection{Hypergeometric Background} Greene \cite{green1,greene} first defined hypergeometric functions over finite fields with Jacobi sums. Initially, this subject focused on establishing transformation and evaluation formulas for Greene's hypergeometric functions over finite fields. Over time the subject has grown immensely with connections to point--counting on algebraic varieties \cite{aop,bcm,long,matrix}, modular forms \cite{a,ao1,hs,tr1,jf,wom}, elliptic curves \cite{k,lennon,mc1,hg,tm,sv,oss,rec}, Galois representations \cite{HMM1,long,whipple}, supercongruences \cite{ao1,nsc,sc1,mor}, graph theory \cite{bgt,pg,pg3}, Sato--Tate distributions \cite{mst,prodst,matrix,oss}, etc.

To define hypergeometric functions over finite fields we begin with hypergeometric data. A multiset $\left\{\alpha, \beta\right\}$, where $\alpha = \{a_{1},\ldots, a_{n}\}$ and $\beta = \{1,b_{2}, \ldots, b_{n}\}$ with $a_{i},b_{i} \in \Q$, is called a \textit{hypergeometric datum}. Furthermore, let $M = \text{lcd}(\alpha \cup \beta)$ be the least positive common denominator of the hypergeometric datum. We now require the following setup over finite fields. Let $p$ be an odd prime, $\zeta_{p}$ a fixed primitive $p$th root of unity, and $q = p^{e}$. Consider the group of multiplicative characters  $\fqhat = \langle \omega \rangle$ on the finite field $\fq$ \footnote{We use the conventions that $\chi(0) = 0$ and $\overline{\chi}$ is the inverse of $\chi$ for any multiplicative character $\chi \in \fqhat$.}. 

 There are many definitions of finite field hypergeometric functions in the literature. We use a version defined by McCarthy \cite{mc1}. Let $\l \in \Q^{\times}$ such that $q \equiv  1 \, (\textrm{mod} \, M)$ and $p \nmid M$. The finite hypergeometric function over $\fq$ is defined as 

\begin{equation}\label{eq:Hqdef}
H_{q}(\alpha, \beta; \l):= \frac{1}{1-q} \sum_{k=0}^{q-2} \prod_{i=1}^{n} \frac{g(\omega^{k+(q-1)a_{i}})g(\omega^{-k-(q-1)b_{i}})}{g(\omega^{(q-1)a_{i}})g(\omega^{-(q-1)b_{i}})} \omega^{k}((-1)^{n}\l), 
\end{equation}
where $$ g(\chi):= \sum_{x \in \fq^{\times}} \chi(x) \zeta_{p}^{\Tr(x)}$$ is the Gauss sum corresponding to the character $\chi$ in $\fqhat$ and $\Tr(x) := x + x^p + \cdots + x^{p^{e-1}}.$ The convention that $H_{q}(\alpha,\beta;0) = 1$ is taken throughout.

An advantage of using the $H_{q}$ function is that the condition of $q \equiv  1 \, (\textrm{mod} \, M)$ can be removed when the hypergeometric datum is defined over $\Q$ \cite{bcm}. A multiset $\alpha = \{a_{1}, \ldots, a_{n}\}$ is defined over $\Q$ if 
$$ \prod_{j=1}^{n} (X - e^{2 \pi i a_{j}}) \in \Z[X].$$

The hypergeometric datum $\{\alpha, \beta\}$ is then defined over $\Q$ if the multisets $\alpha$, $\beta$ are both defined over $\Q$ and $\l \in \Q^{\times}$. Another useful property of hypergeometric data $\{ \alpha, \beta \}$ is to be primitive which means that $a_{i} - b_{j} \notin \Z$ for all $i,j \in [1,n]$. All hypergeometric data considered in this paper is both defined over $\Q$ and primitive. A well-studied example is the datum $\{\{\frac{1}{2}, \frac{1}{2}\},\{1,1\}\}$. Define 

$$H_{q}(\l):= H_{q}\left[\begin{matrix} \frac{1}{2} & \frac{1}{2} \smallskip \\  &1\end{matrix} \; ; \; \l \right] = \frac{1}{1-q} \sum_{\chi \in \fqhat} \frac{g(\phi \chi)^{2}g(\overline{\chi})^{2}}{g(\phi)^{2}} \chi(\l),$$

where $\varepsilon$ and $\phi$ are the trivial and quadratic characters, respectively. Observe that $H_{q}(\l) = -q \cdot \, _{2}F_{1}(\l)_{q}$, for $\l \in \fq^{\times}$, where $_{2}F_{1}(\l)_{q}$ is Greene's hypergeometric function over finite fields for the hypergeometric datum $\{ \{\frac{1}{2}, \frac{1}{2}\}, \{1,1\} \}$. The $H_{p}(\l)$ function and its various normalizations have been studied throughout the literature \cite{ao1,aop,bcm,tr1,k,mc1,sv,wom,rec} and is particularly useful when studying point-counting on algebraic varieties. For example, Koike \cite{k} showed that $H_{q}(\l) = \phi(-1) a_{\l}^{\text{Leg}}(q)$ for all $\l \in \fq \setminus \{0,1\}$, where
$$a_{\l}^{\text{Leg}}(q):= q + 1 - |E_{\l}^{\text{Leg}}(\Fq)|$$
is the trace of Frobenius for the Legendre elliptic curve 
$$E_{\l}^{\text{Leg}}: y^{2} = x(x-1)(x-\l).$$
In \cite{bcm}, Beukers, Cohen, and Mellit generalize the point-counting connection of Koike to hypergeometric data defined over $\Q$ by connecting the $H_{q}$ function for data defined over $\Q$ to point-counting on a certain non-singular completion of an explicit toric variety.

\subsection{Hypergeometric Moments}\label{ss:HGM}

The key connection of Koike allows us to consider the Sato-Tate distribution of the normalized values $H_{q}(\l)/\sqrt{q} \in [-2,2]$ for a fixed $\l \in \fq \setminus \{0,1\}$. In \cite{oss}, Ono, Saad, and Saikia consider all elliptic curves in the Legendre family and determine the distribution of normalized values $H_{q}(\l)/\sqrt{q} \in [-2,2]$, where now both $q$ and $\l \in \fq \setminus \{0,1\}$ vary. They find the distribution is semicircular, explicitly determined in Corollary 1.2 of \cite{oss}, analogous to the classical Sato-Tate conjecture.

Let $m$ be a fixed positive integer. Computing the mean of the $(H_{q}(\l)/\sqrt{q})^{m}$ values over $\fq$ and taking the limit as $p$ tends to infinity gives the first example of a hypergeometric moment. Ono, Saad, and Saikia \cite{oss} determine the values of this limit to be 

\begin{equation}\label{eq:OSS}
\lim_{p \to \infty} \frac{1}{p^{e(\frac{m}{2}+1)}} \sum_{\l \in \mathbb{F}_{p^e}} H_{p^e}(\l)^{m} = \begin{cases}
0 &\quad \textit{if $m$ is odd}\\
C(n) &\quad \textit{if $m$ = 2$n$ is even,}
\end{cases}
\end{equation}
where $m$ and $e$ are fixed positive integers and $C(n) = \frac{(2n)!}{n!(n+1)!}$ denotes the $n$th Catalan number. 

Our first two results extend \eqref{eq:OSS} to the hypergeometric datum $\{\{\frac{1}{3}, \frac{2}{3}\}, \{1,1\} \}$ and determine the even moments for the case $\{ \{\frac{1}{4}, \frac{3}{4}\}, \{1,1\} \}$. We restrict to primes $p$, for simplicity, but the properties of $L$-functions attached to hypergeometric Galois representations \cite{Katzexp} can be used to explore moments for prime powers $q = p^e$. Our first result is as follows.
\begin{theorem}\label{thm:main3}
If $m$ is a fixed positive integer then 
$$\lim_{p \to \infty} \frac{1}{p^{(\frac{m}{2}+1)}} \sum_{\l \in \fp} H_{p}\left[\begin{matrix} \frac{1}{3} & \frac{2}{3} \smallskip \\  &1\end{matrix} \; ; \; \l \right]^{m}= \begin{cases}
0 &\quad \text{if m  is odd}\\
C(n) &\quad \text{if m = 2n is even.}
\end{cases}$$
\end{theorem}

The result in \Cref{thm:main3} was recently established by Ono, Pujahari, Saad, and Saikia \cite{Hess} for primes $p \equiv 1 \pmod{3}$ using an Eicher--Selberg trace formula. We generalize their result to primes $p \equiv 2 \pmod{3}$ in \Cref{thm:main3} using new hypergeometric trace formulas developed by Hoffman, Li, Long, and Tu. See \Cref{sec:traces} for more details on these new trace formulas. Further, our trace formula technique leads to the following new case.

\begin{theorem}\label{thm:main4} If $n$ is a fixed nonnegative integer then
$$\lim_{p \to \infty} \frac{1}{p^{n+1}} \sum_{\l \in \fp} H_{p}\left[\begin{matrix} \frac{1}{4} & \frac{3}{4} \smallskip \\  &1\end{matrix} \; ; \; \l \right]^{2n} = C(n).$$
Further, 
$$\sum_{\l \in \fp} H_{p}\left[\begin{matrix} \frac{1}{4} & \frac{3}{4} \smallskip \\  &1\end{matrix} \; ; \; \l \right]^{2n+1} = 0$$ for primes $p \equiv 5 \pmod{8}$ and $p \equiv 7 \pmod{8}$.
\end{theorem}

Our techniques also allow for the following refinement of \eqref{eq:OSS} to square values.

\begin{theorem}\label{thm:squaremoments}
If $m$ is a fixed positive integer then
$$\lim_{p \to \infty} \frac{1}{p^{(\frac{m}{2}+1)}} \sum_{\l \in \fp} H_{p}(\l^2)^{m} = \begin{cases}
0 &\quad \text{if m is odd}\\
C(n) &\quad \text{if m = 2n is even}
\end{cases}$$
\end{theorem}

The even moments in our main results, the Catalan numbers, are ubiquitous in mathematics. For example, $C(n)$ is the number of permutations of $\{1, \ldots, n\}$ which avoid the permutation pattern $(1 \, 2 \, 3)$. Stanley gives 214 combinatorial interpretations of the Catalan numbers in \cite{stanley}. The Catalan numbers also possess many interesting connections to classical terminating hypergeometric functions. For example, for all $n \geq 0$
$$C(n) = \pFq{2}{1}{1-n&-n}{&2}{1}.$$
A third perspective on the Catalan numbers is as the even moments of the traces in the Lie group $\text{SU}(2)$,
$$\int_{\text{SU}(2)} (\text{Tr} X)^{2n} dX = C(n),$$
where the integral is with respect to the Haar measure on $\text{SU}(2)$. In particular, the even moments in \eqref{eq:OSS} and our main results correspond to the semicircular distribution of traces for the matrices in $\text{SU}(2)$. All moments considered in this paper are related to point counting on appropriate families of elliptic curves, such as the Legendre family for \eqref{eq:OSS}. See \Cref{tab:Etfamily} for more information.

The approach to the proof of \eqref{eq:OSS} by Ono, Saad, and Saikia in \cite{oss} involves combining a result of Schoof \cite{schoof} on counting isomorphism classes of elliptic curves with prescribed torsion requirements, bounds on Hurwitz class numbers, and the use of both holomorphic projection and Rankin--Cohen brackets.

\begin{comment}
The authors use Koike's connection to reduce the moment problem to a question of counting isomorphism classes of certain elliptic curves over $\fq$ with prescribed subgroup requirements. Then a result of Schoof \cite{schoof} gives the count of isomorphism classes for these elliptic curves as a certain weighted sum of Hurwitz class numbers. These connections give an explicit Eichler--Selberg trace formula for the congruence subgroup $\Gamma_{0}(4)$, as used in \cite{a,cy1,lennon}, to establish trace formulas for even weight modular forms on certain subgroups of $\SL_{2}(\Z)$.

The authors then relate the weighted sums of Hurwitz class numbers given by Schoof to the generating function for Hurwitz class numbers, a mock modular form, which is the holomorphic part of Zagier's weight 3/2 nonholomorphic Eisenstein series \cite{zagier}. This connection is realized by considering the holomorphic projection of the Rankin--Cohen bracket of Zagier's weight 3/2 nonholomorphic Eisenstein series with an appropriate theta function. The authors then utilize many tools for studying holomorphic projections involving Hurwitz class numbers developed in the work of Mertens \cite{m2,m1} and Deligne's bound on Fourier coefficients of holomorphic cusp forms to determine the desired asymptotic. 
\end{comment}

Our approach is to use known trace formulas for the $T_{p}$ Hecke operator on spaces of modular forms for congruence subgroups of $\text{SL}_{2}(\Z)$ to obtain a recursion on the moments. We then apply induction on the exponent $m$ using this recursion. The hypergeometric data in \Cref{thm:main3} and \Cref{thm:main4} are related to traces of cusp forms on the groups $\Gamma_{1}(3)$ and $\Gamma_{0}(2)$, respectively.  The methods we use to establish \Cref{thm:main3} and \Cref{thm:main4} can also be applied to a known even weight trace formula for the congruence subgroup $\Gamma_{0}(8)$, established by Ahlgren and Ono \cite{cy1}, to obtain the following square moments for the $H_{p}(\l)$ function in \Cref{thm:squaremoments}.

    A variation of the moments considered in \eqref{eq:OSS}, in the context of McCarthy's $p$-adic hypergeometric functions, was recently studied by Pujahari and Saikia \cite{padic,padic2}. The $\text{SU}(2)$ distribution also arises when counting matrix points on Legendre elliptic curves, as shown by Huang, Ono, and Saad \cite{matrix}. Now given a hypergeometric datum that is defined over $\Q$ and primitive there is no guarantee the moments will converge to a known distribution, such as the $\text{SU}(2)$ distribution. A further restriction on the hypergeometric datum, requiring the data to be algebraic, leads to cases where the $H_{q}$ function can be expressed as a finite sum of character values. These algebraic formulas are finite field analogs of algebraic functions and are established in cases where the even moments do not converge to a known sequence numerically. We now discuss algebraic formulas for hypergeometric data of lengths two and three.

\subsection{Algebraic Formulas in the Finite Field Setting}

The algebraic cases are those in which the solutions to the corresponding hypergeometric differential equation are algebraic functions. See section 3.2 of \cite{long} for background on hypergeometric differential equations.

In \cite{bh} Beukers and Heckman give a simple criteria, called interlacing, to determine if a hypergeometric datum is algebraic.
Suppose that $a_{i}, b_{i} \in \Q$ with $0 \leq a_{1} \leq \cdots \leq a_{n} < 1$ and $0 \leq b_{1} \leq \cdots \leq b_{n} \leq 1$, where the data is consider modulo $\Z$. The two multisets $\alpha = \{a_{1}, \ldots, a_{n}\}$ and $\beta = \{b_{1}, b_{2}, \ldots, b_{n}\}$ are said to interlace if and only if 

\begin{center}
$a_{1}<b_{1}< \cdots < a_{n} < b_{n}$ or $b_{1} < a_{1} < \cdots < b_{n} < a_{n}$.
\end{center}

Beukers and Heckman showed that a hypergeometric datum $\{\alpha, \beta\}$ is algebraic if and only if $\alpha$ and $\beta$ interlace. We first determine algebraic formulas for the three cases of length two data that are defined over $\Q$, primitive, and algebraic. 
\begin{theorem}\label{thm:alg1}
    Let $p$ be an odd prime not dividing $M$ and $\l \in \fp^{\times}$ fixed. Consider the polynomial $f_{\l}(x) = x^{3} + 3x^{2} - 4\l$. Let $N_{f}(\l,p)$ be the number of zeros for $f_{\l}(x)$ in $\fp^{\times}$ counted with multiplicity. If $p > 3$ then
    
    $$(1) \, \,  H_{p}\left[\begin{matrix} \frac{1}{4} & \frac{3}{4}\smallskip \\  &\frac{1}{2} \end{matrix} \; ; \; \l \right] = \bigg(\frac{1 + \phi(\l)}{2} \bigg) \bigg[\phi(1 + \sqrt{\l}) + \phi(1 - \sqrt{\l}) \bigg], $$

    $$(2) \, \,  H_{p}\left[\begin{matrix} \frac{1}{3} & \frac{2}{3} \smallskip \\   & \frac{1}{2} \end{matrix} \; ; \; \l \right] = N_{f}(\l,p) - 1,$$

$$(3) \, \, H_{p}\left[\begin{matrix} \frac{1}{6} & \frac{5}{6} \smallskip \\   & \frac{1}{2} \end{matrix} \; ; \; \l \right] = \phi \bigg(\frac{\l}{3} \bigg)(N_{f}(\l,p) - 1).$$ 
In particular, 
$$
 H_{p}\left[\begin{matrix} \frac{1}{4} & \frac{3}{4}\smallskip \\  &\frac{1}{2} \end{matrix} \; ; \; 1 \right] =\phi(2),   H_{p}\left[\begin{matrix} \frac{1}{3} & \frac{2}{3} \smallskip \\   & \frac{1}{2} \end{matrix} \; ; \; 1 \right] =1, \text{and} \, \,   H_{p}\left[\begin{matrix} \frac{1}{6} & \frac{5}{6} \smallskip \\   & \frac{1}{2} \end{matrix} \; ; \; 1 \right]=\phi(3). 
$$

\end{theorem}

Note the polynomial $f_{\l}(x)$ has no repeated roots for $\l \in \fp \setminus \{0,1\}$. Formula $(1)$ of \Cref{thm:alg1} is a generalization of Theorem 8.11 of \cite{long}, in the case when $A$ is an order four character, to primes $p \equiv 3 \mod 4$. This generalization is possible by the results in \cite{bcm} since the hypergeometric data is defined over $\Q$. Formula $(2)$ is Corollary 1.6 of \cite{bcm}, but we give a direct character sum proof, and formula $(3)$ follows from formula $(2)$.

The analogous result for the three cases of length three data is below.

\begin{theorem}\label{thm:alg2}
    Let $p > 3$ be a prime not dividing $M$, and fix $\l \in \fp^{\times}$. Then

    $$(1) \, \, H_{p}\left[\begin{matrix} \frac{1}{2} & \frac{1}{4} & \frac{3}{4}\smallskip \\  &\frac{1}{3}&\frac{2}{3} \end{matrix} \; ; \; \l \right] = \phi(-3 \l) \sum_{a \in A_{\l,p}} \phi(a) $$

    where $A_{\l,p} = \{a \in \fp \setminus \{0,1\} \, | \, 27 \l(a^{3}-a^{2}) + 4 = 0 \}$.

    $$(2) \, \, H_{p}\left[\begin{matrix} \frac{1}{2} & \frac{1}{6} & \frac{5}{6}\smallskip \\  &\frac{1}{4}&\frac{3}{4} \end{matrix} \; ; \; \l \right] =  \sum_{a \in B_{\l,p}} \phi(a^{2}-a) - \phi \bigg(-\frac{\l}{3} \bigg)$$

    where $B_{\l,p} = \{a \in \fp \setminus \{0,1\} \, | \, 256\l(a^{4}-a^{3}) + 27 = 0  \}$.

    $$(3) \, \,  H_{p}\left[\begin{matrix} \frac{1}{2} & \frac{1}{6} & \frac{5}{6}\smallskip \\  &\frac{1}{3}&\frac{2}{3} \end{matrix} \; ; \; \l \right] = \sum_{a \in C_{\l,p}} \phi(a) $$

    where $C_{\l,p} = \{a \in \fp \setminus \{0,1\} \, | \, (\l-1)a^{3} + 3a^{2}-3a+1=0  \}$.
\end{theorem}

See \Cref{sec:algproofs} for the proofs of \Cref{thm:alg1} and \Cref{thm:alg2}. Note the formulas in \Cref{thm:alg2} can be viewed as analogous to the formulas that Theorem 1.5 of \cite{bcm} give, in a few cases, if the non--singular completion of the affine variety $V_{\l}$, defined in \cite{bcm}, can be explicitly determined. Our formulas allow for efficient computation of the $H_{p}$ function for algebraic data. Further discussions on moments and algebraic formulas for data that is defined over $\Q$ and primitive are given in \Cref{sec:conj}.

\section*{Acknowledgements}
 I thank Ling Long, Esme Rosen, Hasan Saad, and Fang--Ting Tu for their helpful comments and discussions.

\section{The Trace Formulas}\label{sec:traces}
We use $\Tr_{k}(\Gamma,p)$ to denote the trace of the Hecke operator $T_{p}$ on the space of cusp forms $S_{k}(\Gamma)$, where $\Gamma$ is a congruence subgroup of SL$_{2}(\Z)$. Following Frechette, Ono, and Papanikolas \cite{tr1} define 
$$G_{k}(s,p) := \sum_{j=0}^{\frac{k}{2}-1} (-1)^{j} \binom{k-2-j}{j} p^{j} s^{k-2j-2}$$

for $k \geq 2$ even. The proof strategy used on the new results in \Cref{ss:HGM} is to apply induction on the exponent of the moment and use the relevant trace formula to obtain the necessary recursion.

\subsection{Trace Formulas for Subgroups of \mathinhead{\Gamma_{0}(3)}{G3}}

A trace formula for even weight cusp forms on $\Gamma_{0}(3)$ was first obtained by Lennon in Corollary 1.3 of \cite{lennon}. However, this formula is an average of $H_{p^i}(\l)$ values as both $i$ and $\l$ vary. Lennon's formula does not involve the exponent $m$ on the $H_{p}$ function values, which is necessary for our recursive argument. Instead, we use the following result to prove \Cref{thm:main3}.

\begin{theorem}[Hoffman, Li, Long, and Tu \cite{HLLT}]\label{thm:level3trace}
Let $k \geq 1$, $p$ be an odd prime, and $\eta_3$ be any order three character on $\fphat$. Then the following is true.
$$-\Tr_{k+2}(\Gamma_{1}(3),p) = 1 + \phi(-3)^{k} - \delta_{3,k}(p) + \sum_{\l = 2}^{p-1} \sum_{j=0}^{\lfloor \frac{k}{2} \rfloor} (-1)^{j} \binom{k-j}{j} p^{j} H_{p}\left[\begin{matrix} \frac{1}{3} & \frac{2}{3} \smallskip \\  &1\end{matrix} \; ; \; \l \right]^{k-2j} $$
where
$$\delta_{3,k}(p) = \begin{cases}
\displaystyle\sum_{\substack{0 \leq i \leq k \\ k \equiv 2i \, (\text{mod} \, \, 3)}}p^{i} J(\eta_3, \eta_3)^{k-2i}  &\quad \text{if $p \equiv 1 \pmod{3}$}\\

\vspace{2mm}
0&\quad \text{if $p \equiv 2 \pmod{3}$, \quad $k$ odd}\\

\vspace{2mm}
(-p)^{\frac{k}{2}} &\quad \text{if $p \equiv 2 \pmod{3}$, \quad $k$ even}
\end{cases}$$
and $$J(\chi, \psi):= \sum_{x \in \fq^{\times}} \chi(x) \psi(1-x)$$ is a Jacobi sum.
\end{theorem}

\subsection{Trace Formulas for Subgroups of \mathinhead{\Gamma_{0}(2)}{G2}}

The trace formulas needed to prove \Cref{thm:main4}, \Cref{thm:squaremoments}, and give a new proof of \eqref{eq:OSS} are discussed in this section. First note the Legendre elliptic curve, $E_{\l}^{\text{Leg}}$, defined earlier has important connections to $H_{p}(\l)$, as mentioned in the introduction. The analogous curve for the hypergeometric datum $\{\{\frac{1}{2}, \frac{1}{2}, \frac{1}{2}\}, \{1,1,1\}, \l \}$ is the Clausen elliptic curve
$$E_{\l}^{\text{Cl}}: y^{2} = (x-1)(x^{2}+\l)$$
for $\l \notin \left\{-1,0\right\}$. Let $a_{\l}^{\text{Cl}}(p) = p + 1 - |E_{\l}^{\text{Cl}}(\fp)|$ denote the trace of Frobenius for the Clausen curve. The work of Frechette, Ono, and Papanikolas \cite{tr1} gives an even weight trace formula for the group $\Gamma_{0}(2)$ as a polynomial in $a_{\l}^{\text{Cl}}(p)$ and $p$. A variation of their formula, now written in terms of the $H_{p}$ function, was recently established by Hoffman, Li, Long, and Tu \cite{HLLT}.

\begin{comment}
\begin{theorem}[Frechette, Ono, and Papanikolas \cite{tr1}, Theorem 2.3]\label{thm:FOP}
For a prime $p \geq 3$ and $k \geq 2$ even,
$$-\Tr_{k+2}(\Gamma_{0}(2),p) = 2 + \delta_{2,k}(p) + \sum_{\l=1}^{p-2} G_{k+2}(a_{\l}^{\textnormal{Cl}}(p), p)$$
where 
$$\delta_{2,k}(p) = \begin{cases}
\frac{1}{2} \bigg[G_{k+2}(2a,p) + G_{k+2}(2b,p) \bigg]&\quad \text{if $p \equiv 1 \pmod{4}$}\\
(-p)^{k}&\quad \text{if $p \equiv 3 \pmod{4}$}
\end{cases}$$

with $p=a^{2}+b^2$ when $p \equiv 1 \pmod{4}$, where $a,b \geq 0$ with $a$ odd.
\end{theorem}
\end{comment}

\begin{theorem}[Hoffman, Li, Long, and Tu \cite{HLLT}]\label{thm:newlevel2}
Let $r \geq 1$ and $p$ is an odd prime. Then the following are true.
$$\Tr_{2r+2}(\Gamma_{0}(2),p) = -2 - \delta_{4,2r}(p) - \sum_{\l = 2}^{p-1} \sum_{j=0}^{r} (-1)^{j} \binom{2r-j}{j} p^{j} H_{p}\left[\begin{matrix} \frac{1}{4} & \frac{3}{4} \smallskip \\  &1\end{matrix} \; ; \; \l \right]^{2(r-j)} $$

where $\eta_{4}$ is any order four character and

$$\delta_{4,2r}(p) = \begin{cases}
\displaystyle{p^{r} \sum_{-\frac{r}{2} \leq i \leq \frac{r}{2}} \bigg(\frac{J(\eta_{4}, \eta_{4})^{2}}{p} \bigg)^{2i}}&\quad \text{if $p \equiv 1 \pmod{4}$}\\
(-p)^{r}&\quad \text{if $p \equiv 3 \pmod{4}$}\\
\end{cases}$$
\end{theorem}

\Cref{thm:newlevel2} is derived from \cite{HLLT} as follows. Choosing the datum $\left\{\{\frac{1}{2},\frac{1}{2},\frac{1}{2}\},\{1,1,1\}\right\}$ in Theorem 1 of \cite{HLLT} then using Theorem 7.1 and Proposition 8.5 from \cite{long} gives the result when $p \equiv 1 \pmod{4}$. Then the result when $p \equiv 1 \pmod{4}$ be extended to all odd primes by following the proof of Lemma 5 in \cite{HLLT}, now using the other Clausen formula from Theorem 7.1 of \cite{long}.

We use \Cref{thm:newlevel2} to prove \Cref{thm:main4}, the even moments for the hypergeometric datum $\{\{\frac{1}{4}, \frac{3}{4}\}, \{1,1\} \}$. The following theorems of Ahlgren and Ahlgren--Ono provides even weight trace formulas for the groups $\Gamma_{0}(4)$ and $\Gamma_{0}(8)$ and are used to prove \Cref{thm:squaremoments}.

\begin{theorem}[(1) Ahlgren \cite{a}; (2) Ahlgren--Ono \cite{cy1}] \label{thm:squaretraces} If $p$ is an odd prime, and $k \geq 2$ is even, then the following are true.

$$(1) \, \, \, -\Tr_{k+2}(\Gamma_{0}(4),p) = 3 + \sum_{\l=2}^{p-1} G_{k+2}(a_{\l}^{\textnormal{Leg}}(p), p)$$

$$(2) \, \, \, -\Tr_{k+2}(\Gamma_{0}(8),p) = 4 + \sum_{\l=2}^{p-2} G_{k+2}(a_{\l^2}^{\textnormal{Leg}}(p),p).$$

\end{theorem}

Both formulas in \Cref{thm:squaretraces} can be rewritten in terms of $H_{p}(\l)$ using the result of Koike. Recently, analogous trace formulas, involving $p$--adic hypergeometric functions, have been established by Pujahari and Saikia \cite{padic} for the groups $\Gamma_{0}(4)$ and $\Gamma_{0}(8)$. Furthermore, new Eichler--Selberg trace formulas, involving Chebyshev polynomials and Hurwitz class numbers, were recently established by Ono and Saad \cite{es} for the groups $\Gamma_{0}(2)$ and $\Gamma_{0}(4)$.

We have observed that, in certain cases, odd moments correspond to trace formulas for modular forms of odd weight and that even moments correspond to trace formulas for modular forms of even weight. For example, formula $(1)$ of \Cref{thm:squaretraces} was recently generalized to include odd-weight modular forms on $\Gamma_{1}(4)$ in the following result. 

\begin{theorem}[Hoffman, Li, Long, and Tu \cite{HLLT}]\label{thm:level4trace}
Let $k \geq 1$ and $p$ be an odd prime. Then the following is true.
$$-\Tr_{k+2}(\Gamma_{1}(4),p) = \frac{1 + (-1)^{k}}{2} + 1 + \phi(-1)^k + \sum_{\l = 2}^{p-1} \sum_{j=0}^{\lfloor \frac{k}{2} \rfloor} (-1)^{j} \binom{k-j}{j} p^{j} H_{p}(\l)^{k-2j}$$
\end{theorem}

Note that \Cref{thm:level4trace} can be used to give an alternate proof of \eqref{eq:OSS}.

\subsection{Bounds on Coefficients of Modular Forms}

A key part of our inductive method involves bounding coefficients of modular forms. Let $\tau \in \mathcal{H}$, where $\mathcal{H} = \{\tau \in \mathbb{C} \, | \, \Im(\tau) > 0\}$, and write $q:= e^{2 \pi i \tau}$. The following well-known result of Deligne allows us to obtain bounds for the Fourier coefficients of cusp forms.

\begin{proposition}[Cohen and Stromburg \cite{mf}, Remark 9.3.15.]\label{prop:Deligne}
If $\displaystyle f = \sum_{n \geq 1} a_{n}q^{n}$ is a cusp form of integer weight $k$ for a congruence subgroup, then for all $\varepsilon> 0$ we have $a_{n} = O(n^{\frac{k-1}{2}+\varepsilon})$.
\end{proposition}

\section{Hypergeometric Character Sums}\label{sec:charsums}

\subsection{Gauss and Jacobi Sums} We begin by recalling various properties of Gauss and Jacobi sums needed to prove \Cref{thm:alg1} and \Cref{thm:alg2}.
A more detailed treatment of Gauss and Jacobi sums can be found in \cite{jacobi, long, ir}.
If $\chi \in \fqhat$ we define 

$$\delta(x) = \begin{cases}
1 &\quad \text{if $x=0$},\\
0 &\quad \text{if $x \neq 0$}
\end{cases} \quad \quad {\delta(\chi) = \begin{cases}
1 &\quad \text{if $\chi = \eps$},\\
0 &\quad \text{if $\chi \neq \eps$}.
\end{cases}}$$

We rely on the following properties of Gauss sums.

\begin{proposition}[\cite{jacobi}]\label{prop:gaussprops}
Let $\chi, \psi \in \fqhat$ and $m \in \N$ with $m \mid q-1$. Then
\\

$(1) \, g(\eps) = -1.$
\\

$(2) \, g(\chi)g(\overline{\chi}) = \chi(-1)q - (q-1) \delta(\chi).$
\\

$(3)$ [The Hasse--Davenport Relation] For any multiplicative character $\psi \in \fqhat$, we have

$$\prod_{\substack{\chi \in \fqhat \\ \chi^{m} = \eps}} g(\chi \psi) = -g(\psi^{m})\psi(m^{-m}) \prod_{\substack{\chi \in \fqhat \\ \chi^{m} = \eps}}g(\chi).$$

\end{proposition}

We rely on the following well--known properties of Jacobi sums and their connection to Gauss sums.

\begin{proposition}[\cite{jacobi,long,ir}]\label{prop:jacobiprops} If $\chi, \psi \in \fqhat$ then the following are true.

$\displaystyle (1) \, J(\chi,\psi) = \frac{g(\chi)g(\psi)}{g(\chi\psi)} + (q-1)\psi(-1) \delta(\chi\psi),$
\\

$(2) \, J(\chi, \eps) = -1 + (q-1)\delta(\chi),$
\\

$(3) \, J(\chi, \overline{\chi}) = -\chi(-1) + (q-1) \delta(\chi),$
\\

$(4) \, J(\chi,\chi) = \overline{\chi}(4) J(\chi, \phi),$
\\

$(5) \, J(\chi, \overline{\psi}) = \chi(-1) J(\chi, \psi \overline{\chi}),$
\\

$(6)$ If $\chi,\psi, \chi\psi \neq \eps$ then $|J(\chi,\psi)| = \sqrt{q}.$
\end{proposition}

We now recall relevant identities for hypergeometric character sums.

\subsection{Hypergeometric Character Sum Identities}

Consider the hypergeometric datum $\text{HD}_{d}:= \left\{\{1/d,1-1/d\},\{1,1\}\right\}$. The $H_{p}(\text{HD}_{d};\l)$ functions naturally appear in point counts on certain families of elliptic curves when $d \in \{2,3,4,6\}$. Let $\widetilde{E}_{d}(\l)$ denote the family of elliptic curves related to the $H_{p}(\text{HD}_{d};\l)$ functions when $\l \notin \{0,1,\infty\}$. Further, we record the monodromy groups $\Gamma$ associated with each family of elliptic curves $\widetilde{E}_{d}(\l)$. We only record the monodromy groups which are congruence.

\renewcommand{\arraystretch}{1.6}
\begin{table}[h]
\caption{The $\widetilde{E}_{d}(\l)$ Families}
\label{tab:Etfamily}
\begin{tabular}{|c|c|c|}
\hline
$d$ & $\widetilde{E}_{d}(\l)$ & $\Gamma$\\
\hline
$2$ & $y^{2} = x(x-1)(x - \l)$ &  $\Gamma(2)$\\
\hline
$3$ & $y^{2} + xy + \frac{\l}{27}y = x^3$ & $\Gamma_{0}(3)$\\
\hline
$4$ & $y^{2} = x \left(x^{2} + x + \frac{\l}{4} \right)$ & $\Gamma_{0}(2)$\\
\hline
$6$ & $y^{2} + xy = x^{3} - \frac{\l}{432}$ & $-$\\
\hline
\end{tabular}
\end{table}
\renewcommand{\arraystretch}{1}

The $d = 2$ case appears in Koike's result, discussed earlier. We now recall the generalization of Koike's point-counting result to the $d = 3,4,6$ cases for the elliptic curves $\widetilde{E}_{d}(\l)$ given in \Cref{tab:Etfamily} for the convenience of the reader.
\begin{lemma}[\cite{bcm,HLLT}] Suppose $q = p^{e}$ with $p \geq 5$ prime and $\l \in \fq \setminus \{0,1\}$ such that $\widetilde{E}_{d}(\l)$ in \Cref{tab:Etfamily} is an elliptic curve defined over $\fq$. Let $|\widetilde{E}_{d}(\l)(\fq)|$ denote the number of $\fq$ points on the curves $\widetilde{E}_{d}(\l)$ in \Cref{tab:Etfamily}. Then 
\begin{equation}\label{eq:Ed_count}
|\widetilde{E}_{d}(\l)(\fq)| = q+1 - H_{p}\left[\begin{matrix} \frac{1}{d} & 1 - \frac{1}{d}\smallskip \\  &1 \end{matrix} \; ; \; \l \right].
\end{equation}
\end{lemma}

We now recall several relevant identities for the $H_{p}$ function.

\begin{lemma}[Lemma 4 of \cite{HLLT}, Theorem 8.6 of \cite{RSwin}, and Theorem 4.4 of \cite{greene}]\label{lem:2F1transform}
Let $d \in \{2,3,4,6\}$ and $\kappa_{2} = \kappa_{6} = - 1$, $\kappa_{3} = -3, \kappa_{4} = -2$. If $p > 5$ is prime and $\l \in \fp \setminus \{0,1\}$ then the following identities hold.
\begin{equation}\label{eq:2F1transform}
H_{p}\left[\begin{matrix} \frac{1}{d} & 1 - \frac{1}{d}\smallskip \\  &1 \end{matrix} \; ; \; \l \right] = \left(\frac{\kappa_{d}}{p}\right) \, H_{p}\left[\begin{matrix} \frac{1}{d} & 1 - \frac{1}{d}\smallskip \\  &1 \end{matrix} \; ; \; 1 - \l \right],
\end{equation}
\begin{equation}\label{eq:2F1middle}
\text{If} \quad \, \left(\frac{\kappa_{d}}{p}\right) = -1 \quad \text{then} \quad H_{p}\left[\begin{matrix} \frac{1}{d} & 1 - \frac{1}{d}\smallskip \\  &1 \end{matrix} \; ; \; \frac{1}{2} \right] = 0 ,
\end{equation}
\begin{equation}\label{eq:2F1one}
H_{p}\left[\begin{matrix} \frac{1}{d} & 1 - \frac{1}{d}\smallskip \\  &1 \end{matrix} \; ; \; 1 \right] = \left(\frac{\kappa_{d}}{p}\right),
\end{equation}
and
\begin{equation}\label{eq:Greenetransform}
H_{p}(\l) = \phi(\l)H_{p}(1/\l).
\end{equation}
\end{lemma}

An important corollary of \Cref{lem:2F1transform} is determining the odd moment sum for half of the relevant primes in the $d = 4,6$ cases.

\begin{corollary}\label{cor:oddmoments}
If $d \in \{4,6\}$ and $n \geq 0$ then
$$\sum_{\l \in \fp} H_{p}\left[\begin{matrix} \frac{1}{d} & 1 - \frac{1}{d}\smallskip \\  &1 \end{matrix} \; ; \; \l \right]^{2n+1} = 0$$ for primes $p \equiv d + 1 \pmod{2d}$ and $p \equiv 2d - 1 \pmod{2d}$.
\end{corollary}

\begin{proof} Consider the $(\l,1-\l)$ pairs in $\fp^{2}$. If $\l \notin \left\{0,1,\frac{1}{2}\right\}$ then 
$$H_{p}\left[\begin{matrix} \frac{1}{d} & 1 - \frac{1}{d}\smallskip \\  & 1 \end{matrix} \; ; \; \l \right] = - H_{p}\left[\begin{matrix} \frac{1}{d} & 1 - \frac{1}{d}\smallskip \\  & 1 \end{matrix} \; ; \; 1 - \l \right],$$ 
$$H_{p}\left[\begin{matrix} \frac{1}{d} & 1 - \frac{1}{d}\smallskip \\  & 1 \end{matrix} \; ; \; 0 \right] + H_{p}\left[\begin{matrix} \frac{1}{d} & 1 - \frac{1}{d}\smallskip \\  & 1 \end{matrix} \; ; \; 1 \right] = 0,$$
and
$$H_{p}\left[\begin{matrix} \frac{1}{d} & 1 - \frac{1}{d}\smallskip \\  & 1 \end{matrix} \; ; \; \frac{1}{2} \right] = 0$$ by \Cref{lem:2F1transform}. Therefore,
\begin{equation*}
\begin{split}
\sum_{\l \in \fp} H_{p}\left[\begin{matrix} \frac{1}{d} & 1 - \frac{1}{d}\smallskip \\  & 1 \end{matrix} \; ; \; \l \right]^{2n+1} &= \sum_{(\l,1-\l) \in \fp^{2}} H_{p}\left[\begin{matrix} \frac{1}{d} & 1 - \frac{1}{d}\smallskip \\  & 1 \end{matrix} \; ; \; \l \right]^{2n+1}\\
&= 0.
\end{split}
\end{equation*}
\end{proof}

The sums we consider for the moments can be indexed over all of $\fp$ since $H_{p}(0) = 1$ and $H_{p}(1)$ is $\pm 1$, by assumption and \Cref{lem:2F1transform}. Hence, the terms at $\l = 0,1$ vanish after normalizing by the appropriate power of $p$ and taking the limit over primes $p$. We now proceed to prove the moments.

\section{The Level 3 Case}

In this section, we prove \Cref{thm:main3}. Let 

$$c(k) = 1 + \phi(-3)^{k} - \delta_{3,k}(p).$$

If $p \equiv 1 \pmod{3}$ then $J(\eta_{3}, \eta_{3}) = O(\sqrt{p})$, by formula $(6)$ of \Cref{prop:jacobiprops} so $c(k) = O(p^{\frac{k}{2}})$. Now suppose $p \equiv 2 \pmod{3}$. If $k$ is even then $c(k) = O(p^{\frac{k}{2}})$ and if $k$ is odd then $c(k) = 0$ by \Cref{thm:level3trace}. The normalizing factor for the moments is $p^{-\frac{k}{2}-1}$ so the normalized $c(k)$ term will vanish in the limit over primes $p$ for all cases.

\subsection{The Odd Moments}

We now determine the odd moments. Letting $k=2r+1$ in \Cref{thm:level3trace} yields
\begin{equation*}
\begin{split}
-\Tr_{2r+3}(\Gamma_{1}(3),p) &= c(2r+1) + \sum_{j=0}^{r} (-1)^{j} \binom{(2r+1)-j}{j} p^{j} \sum_{\l=2}^{p-1} H_{p}\left[\begin{matrix} \frac{1}{3} & \frac{2}{3} \smallskip \\  &1\end{matrix} \; ; \; \l \right]^{2(r-j)+1}\\
&= c(2r+1) + \sum_{\l=2}^{p-1} H_{p}\left[\begin{matrix} \frac{1}{3} & \frac{2}{3} \smallskip \\  &1\end{matrix} \; ; \; \l \right]^{2r+1}\\
&+ \sum_{j=1}^{r} (-1)^{j} \binom{(2r+1)-j}{j} p^{j} \sum_{\l=2}^{p-1} H_{p}\left[\begin{matrix} \frac{1}{3} & \frac{2}{3} \smallskip \\  &1\end{matrix} \; ; \; \l \right]^{2(r-j)+1}.
\end{split}
\end{equation*}

Then

\begin{equation}\label{eq:3odd}
\begin{split}
\sum_{\l=2}^{p-1} H_{p}\left[\begin{matrix} \frac{1}{3} & \frac{2}{3} \smallskip \\  &1\end{matrix} \; ; \; \l \right]^{2r+1}
&= -c(2r+1) - \Tr_{2r+3}(\Gamma_{1}(3),p)  \\
&-\sum_{j=1}^{r} (-1)^{j} \binom{(2r+1)-j}{j} p^{j} \sum_{\l=2}^{p-1} H_{p}\left[\begin{matrix} \frac{1}{3} & \frac{2}{3} \smallskip \\  &1\end{matrix} \; ; \; \l \right]^{2(r-j)+1}.
\end{split}
\end{equation}

We proceed by induction. The base case $r = 0$ follows easily since $$ \sum_{\l = 2}^{p-1} H_{p}\left[\begin{matrix} \frac{1}{3} & \frac{2}{3} \smallskip \\  &1\end{matrix} \; ; \; \l \right] = O(1)$$ using the orthogonality of characters and the definition of the $H_{p}$ function from \eqref{eq:Hqdef}.  Now consider the inductive step. Normalizing \eqref{eq:3odd} gives

$$ \frac{1}{p^{r+\frac{3}{2}}} \sum_{\l \in \fp} H_{p}\left[\begin{matrix} \frac{1}{3} & \frac{2}{3} \smallskip \\  &1\end{matrix} \; ; \; \l \right]^{2r+1}$$

$$= O \bigg(\frac{1}{\sqrt{p}}\bigg) - \sum_{j=1}^{r} (-1)^{j} \binom{(2r+1)-j}{j} \Bigg(\frac{1}{p^{r-j+\frac{3}{2}}} \sum_{\l \in \fp} H_{p}\left[\begin{matrix} \frac{1}{3} & \frac{2}{3} \smallskip \\  &1\end{matrix} \; ; \; \l \right]^{2(r-j)+1} \Bigg).$$

Taking limits yields

$$\lim_{p \to \infty} \frac{1}{p^{r+\frac{3}{2}}} \sum_{\l \in \fp} H_{p}\left[\begin{matrix} \frac{1}{3} & \frac{2}{3} \smallskip \\  &1\end{matrix} \; ; \; \l \right]^{2r+1}$$

$$= - \sum_{j=1}^{r} (-1)^{j} \binom{(2r+1)-j}{j} \Bigg( \lim_{p \to \infty} \frac{1}{p^{r-j+\frac{3}{2}}} \sum_{\l \in \fp} H_{p}\left[\begin{matrix} \frac{1}{3} & \frac{2}{3} \smallskip \\  &1\end{matrix} \; ; \; \l \right]^{2(r-j)+1} \Bigg) = 0$$

since 

$$\lim_{p \to \infty} \frac{1}{p^{r-j+\frac{3}{2}}} \sum_{\l \in \fp} H_{p}\left[\begin{matrix} \frac{1}{3} & \frac{2}{3} \smallskip \\  &1\end{matrix} \; ; \; \l \right]^{2(r-j)+1} = 0$$

for each $1 \leq j \leq r$, by the inductive hypothesis.

\subsection{The Even Moments}

We now show the even moments. In our cases, the inductive step is realized as a special case of the following result of Chu.

\begin{lemma}[Chu \cite{cat}, Corollary 2]\label{lem:chu}
Let $m,n \in \mathbb{N} \cup \{0\}$. Then

$$ \sum_{j=0}^{\lfloor \frac{m}{2} \rfloor} (-1)^{j} \binom{m - j}{j} C(n-j) = \begin{cases} 
      \binom{2n-m}{n} &\quad \text{if} \quad m > 2n\\
      \\
      \binom{2n-m}{n} \cdot \frac{m+1}{n+1} &\quad \text{if} \quad m \leq 2n. \\
   \end{cases}$$
\end{lemma}

Letting $k = 2r$ in \Cref{thm:level3trace} gives

\begin{equation*}
\begin{split}
&-\Tr_{2r+2}(\Gamma_{1}(3),p)\\
&= c(2r) + \sum_{\l=2}^{p-1} H_{p}\left[\begin{matrix} \frac{1}{3} & \frac{2}{3} \smallskip \\  &1\end{matrix} \; ; \; \l \right]^{2r} + \sum_{j=1}^{r} (-1)^{j} \binom{2r - j}{j} p^{j} \sum_{\l=2}^{p-1} H_{p}\left[\begin{matrix} \frac{1}{3} & \frac{2}{3} \smallskip \\  &1\end{matrix} \; ; \; \l \right]^{2(r - j)}
\end{split}
\end{equation*}

so

\begin{equation*}
\begin{split}
&\sum_{\l=2}^{p-1} H_{p}\left[\begin{matrix} \frac{1}{3} & \frac{2}{3} \smallskip \\  &1\end{matrix} \; ; \; \l \right]^{2r}\\
&= -\Tr_{2r +2}(\Gamma_{1}(3),p) - c(2r) - \sum_{j=1}^{r} (-1)^{j} \binom{2r - j}{j} p^{j} \sum_{\l=2}^{p-1} H_{p}\left[\begin{matrix} \frac{1}{3} & \frac{2}{3} \smallskip \\  &1\end{matrix} \; ; \; \l \right]^{2(r - j)}.
\end{split}
\end{equation*}

After normalization we have

\begin{equation}\label{eq:3evennormalized}
\begin{split}
&\frac{1}{p^{r + 1}} \sum_{\l \in \fp} H_{p}\left[\begin{matrix} \frac{1}{3} & \frac{2}{3} \smallskip \\  &1\end{matrix} \; ; \; \l \right]^{2r}\\
&= O \bigg( \frac{1}{\sqrt{p}} \bigg) -  \sum_{j=1}^{r} (-1)^{j} \binom{2r - j}{j} \Bigg( \frac{1}{p^{r -j +1}} \sum_{\l \in \fp} H_{p}\left[\begin{matrix} \frac{1}{3} & \frac{2}{3} \smallskip \\  &1\end{matrix} \; ; \; \l \right]^{2(r - j)} \Bigg).
\end{split}
\end{equation}

Now consider the even moments. We apply induction on $r$. The base case is $r = 1$ which gives 
$$\lim_{p \to \infty} \frac{1}{p^2} \sum_{\l \in \fp} H_{p}\left[\begin{matrix} \frac{1}{3} & \frac{2}{3} \smallskip \\  &1\end{matrix} \; ; \; \l \right]^{2} = 1 = C(1)$$
by taking the limit of \eqref{eq:3evennormalized} with $r = 1$. Suppose the result holds for all $1 \leq k \leq r$. Now taking the limit of \eqref{eq:3evennormalized} for $r = k+1$ and using the inductive hypothesis gives

 \begin{equation*}
 \begin{split}
 &\lim_{p \to \infty} \frac{1}{p^{k+2}} \sum_{\l \in \fp} H_{p}\left[\begin{matrix} \frac{1}{3} & \frac{2}{3} \smallskip \\  &1\end{matrix} \; ; \; \l \right]^{2k+2}\\
 &= - \sum_{j=1}^{k+1} (-1)^{j} \binom{(2k+2)-j}{j} \bigg(\lim_{p \to \infty} \frac{1}{p^{k-j+2}} \sum_{\l \in \fp} H_{p}\left[\begin{matrix} \frac{1}{3} & \frac{2}{3} \smallskip \\  &1\end{matrix} \; ; \; \l \right]^{2(k-j+1)} \bigg)\\
 &= - \sum_{j=1}^{k+1} (-1)^{j} \binom{(2k+2)-j}{j} C(k+1-j).
 \end{split}
 \end{equation*}

 Now letting $m = 2k + 2$ and $n = k + 1$ in \Cref{lem:chu} gives

\begin{comment}
 $$ \sum_{j=1}^{k+1} (-1)^{j} \binom{(2k+2)-j}{j} C(k+1-j) + C(k+1) = 0$$

 so
 \end{comment}

 $$-\sum_{j=1}^{k+1} (-1)^{j} \binom{(2k+2)-j}{j} C(k+1-j) = C(k+1).$$

 Therefore,

$$\lim_{p \to \infty} \frac{1}{p^{k+2}} \sum_{\l \in \fp} H_{p}\left[\begin{matrix} \frac{1}{3} & \frac{2}{3} \smallskip \\  &1\end{matrix} \; ; \; \l \right]^{2k+2} = C(k+1).$$

\section{The Level 2 Case}

In this section, we prove \Cref{thm:main4}. It is clear that $\delta_{4,2r}(p) = O(p^r)$, in \Cref{thm:newlevel2}, since $J(\eta_{4}, \eta_{4}) = O(\sqrt{p})$. Therefore,

\begin{equation*}
    \begin{split}
        &-\Tr_{2r+2}(\Gamma_{0}(2),p)\\
        &= O(p^r) + \sum_{\l=2}^{p-2} H_{p}\left[\begin{matrix} \frac{1}{4} & \frac{3}{4}\smallskip \\  &1 \end{matrix} \; ; \; \l \right]^{2r} + \sum_{j=1}^{r} (-1)^{j} \binom{2r-j}{j} p^{j} \sum_{\l=2}^{p-2} H_{p}\left[\begin{matrix} \frac{1}{4} & \frac{3}{4}\smallskip \\  &1 \end{matrix} \; ; \; \l \right]^{2(r-j)}
    \end{split}
\end{equation*}

so 

\begin{equation*}
\begin{split}
&\sum_{\l=2}^{p-2} H_{p}\left[\begin{matrix} \frac{1}{4} & \frac{3}{4}\smallskip \\  &1 \end{matrix} \; ; \; \l \right]^{2r}\\
&= -\Tr_{2r+2}(\Gamma_{0}(2),p) - O(p^r) - \sum_{j=1}^{r} (-1)^{j} \binom{2r-j}{j} p^{j} \sum_{\l=2}^{p-2} H_{p}\left[\begin{matrix} \frac{1}{4} & \frac{3}{4}\smallskip \\  &1 \end{matrix} \; ; \; \l \right]^{2(r-j)}\\
&= O(p^{r+ \frac{1}{2}}) - \sum_{j=1}^{r} (-1)^{j} \binom{2r-j}{j} p^{j} \sum_{\l=2}^{p-2} H_{p}\left[\begin{matrix} \frac{1}{4} & \frac{3}{4}\smallskip \\  &1 \end{matrix} \; ; \; \l \right]^{2(r-j)},
\end{split}
\end{equation*}

by \Cref{prop:Deligne}. Normalizing gives 

\begin{equation}
\begin{split}
\frac{1}{p^{r+1}} \sum_{\l=2}^{p-2} H_{p}\left[\begin{matrix} \frac{1}{4} & \frac{3}{4}\smallskip \\  &1 \end{matrix} \; ; \; \l \right]^{2r} &= O\bigg( \frac{1}{\sqrt{p}}\bigg) - \sum_{j=1}^{r} (-1)^{j} \binom{2r-j}{j} \frac{1}{p^{r-j+1}} H_{p}\left[\begin{matrix} \frac{1}{4} & \frac{3}{4}\smallskip \\  &1 \end{matrix} \; ; \; \l \right]^{2(r-j)}.
\end{split}
\end{equation}

Applying induction on $r$, with the same inductive step as the proof of \Cref{thm:main3}, gives the even moments. The result for odd exponents in \Cref{thm:main4} is the $d = 4$ case of \Cref{cor:oddmoments}.

\section{The Level 8 Case}

In this section, we prove \Cref{thm:squaremoments}.

\subsection{The Odd Moments}

The odd moments in \Cref{thm:squaremoments} follow from the odd moments in \eqref{eq:OSS} by considering square and non-square values. Let $m = 2r+1$. The odd moments of \eqref{eq:OSS}, in the case $q=p$, state that

$$\lim_{p \to \infty} \frac{1}{p^{r+ \frac{3}{2}}} \sum_{\l \in \mathbb{F}_{p}} H_{p}(\l)^{2r+1} = 0.$$
Using \eqref{eq:Greenetransform} gives
\begin{equation*}
\begin{split}
\sum_{\substack{\l \in \mathbb{F}_{p} \\\phi(\l)=-1}} H_{p}(\l)^{2r+1} &= \sum_{\substack{(\l, \frac{1}{\l}) \in (\mathbb{F}_{p}^{\times})^{2} \\ \phi(\l)=-1}} \bigg[ H_{p}(\l)^{2r+1} + H_{p}\bigg(\frac{1}{\l} \bigg)^{2r+1}\bigg]\\
&= \sum_{\substack{(\l, \frac{1}{\l}) \in (\mathbb{F}_{p}^{\times})^{2} \\ \phi(\l)=-1}} [ H_{p}(\l)^{2r+1} - H_{p}(\l)^{2r+1}]\\
&= 0.
\end{split}
\end{equation*}

Now observe that
$$\sum_{\l \in \fp^{\times}} H_{p}(\l^2)^{2r+1} = 2 \sum_{\substack{\l \in \mathbb{F}_{p}^{\times} \\ \phi(\l)=1}} H_{p}(\l)^{2r+1}$$

so the odd moments in \eqref{eq:OSS} imply
$$\lim_{p \to \infty} \frac{1}{p^{r+ \frac{3}{2}}} \sum_{\l \in \mathbb{F}_{p}} H_{p}(\l^2)^{2r+1} = 0.$$

\subsection{The Even Moments}
Now consider the even moments. We first rewrite the trace formula $(2)$ of \Cref{thm:squaretraces} in terms of the $H_{p}(\l)$ function. Observe that $a_{p}(\l^2) = \phi(-1)H_{p}(\l^2)$, for $\l \in \fp \setminus \{0,1\}$, by Koike's result. Letting $k = 2r$ in formula $(2)$ of \Cref{thm:squaretraces} gives

\begin{equation*}
\begin{split}
-\Tr_{2r+2}(\Gamma_{0}(8),p) &= 4 + \sum_{j=0}^{r} (-1)^{j} \binom{2r-j}{j} p^{j} \sum_{\l=2}^{p-2} H_{p}(\l^2)^{2(r-j)} \\
&= 4 + \sum_{\l=2}^{p-2} H_{p}(\l^2)^{2r} + \sum_{j=1}^{r} (-1)^{j} \binom{2r-j}{j} p^{j} \sum_{\l=2}^{p-2} H_{p}(\l^2)^{2(r-j)}.\\
\end{split}
\end{equation*}
Then
$$\sum_{\l=2}^{p-2} H_{p}(\l^2)^{2r} = -\Tr_{2r+2}(\Gamma_{0}(8),p) -4 - \sum_{j=1}^{r} (-1)^{j} \binom{2r-j}{j} p^{j} \sum_{\l=2}^{p-2} H_{p}(\l^2)^{2(r-j)}$$
so
\begin{equation}\label{eq:normeven}
\frac{1}{p^{r + 1}} \sum_{\l \in \fp} H_{p}(\l^2)^{2r} = O \bigg( \frac{1}{\sqrt{p}} \bigg) -  \sum_{j=1}^{r} (-1)^{j} \binom{2r - j}{j} \bigg[ \frac{1}{p^{r -j +1}} \sum_{\l \in \fp} H_{p}(\l^2)^{2(r - j)} \bigg]. 
\end{equation}

The even moments for \Cref{thm:squaremoments} follow from \eqref{eq:normeven} using the same inductive argument used in the proofs of \Cref{thm:main3} and \Cref{thm:main4}. 

\section{Proofs of \Cref{thm:alg1} and \Cref{thm:alg2}}\label{sec:algproofs}

We now prove \Cref{thm:alg1} and \Cref{thm:alg2}. In this section, every sum over $\chi$ refers to the sum over all $\chi$ in $\fphat$, unless indicated otherwise.
\begin{proof}[The Proof of \Cref{thm:alg1}]
    We show formula $(2)$ as formula $(1)$ follows in a similar way and formula $(3)$ is obtained by letting $\chi \mapsto \phi \chi$ in formula $(2)$. We have
\begin{equation*}
\begin{split}
     H_{p}\left[\begin{matrix} \frac{1}{3} & \frac{2}{3}\smallskip \\  &\frac{1}{2} \end{matrix} \; ; \; \l \right] 
     &= \frac{1}{p-1} \sum_{\chi} \frac{g(\chi^3)g(\overline{\chi^2})}{g(\chi)} \chi \bigg(\frac{4\l}{27}\bigg)\\
     &= \frac{1}{p-1} \sum_{\chi} \bigg[J(\chi^{3}, \overline{\chi^{2}}) - (p-1)\delta(\chi) \bigg] \chi \bigg(\frac{4\l}{27} \bigg)\\
     &= \frac{1}{p-1} \sum_{\chi} \bigg[J(\overline{\chi^{2}}, \overline{\chi}) - (p-1)\delta(\chi) \bigg] \chi \bigg(\frac{4\l}{27} \bigg)\\
     &= \frac{1}{p-1} \sum_{a \in \fp \setminus \{0,1\}} \bigg[\sum_{\chi} \chi \bigg(\frac{4\l}{27a^{2}(1-a)} \bigg) \bigg] - 1 \\
     &= N_{f}(\l,p) - 1,
\end{split}
\end{equation*}
where $N_{f}(\l,p)$ is the number of solutions to $27x^{3}-27x^{2}+4\l = 0$ in $\fp \setminus \{0,1\}$ by \Cref{prop:gaussprops} and \Cref{prop:jacobiprops}. Then letting $x \mapsto -x/3$ gives the result. Now to show formula $(3)$ first observe that
\begin{equation*}
\begin{split}
 H_{p}\left[\begin{matrix} \frac{1}{6} & \frac{5}{6}\smallskip \\  &\frac{1}{2} \end{matrix} \; ; \; \l \right] 
     &= \frac{1}{p-1} \sum_{\chi} \frac{g(\phi \chi^{3})g(\overline{\chi^2})}{g(\phi \chi)} \chi \bigg(\frac{4\l}{27} \bigg).
\end{split}
\end{equation*}
Letting $\chi \mapsto \phi \chi$ gives
\begin{equation*}
\begin{split}
       \frac{1}{p-1} \sum_{\chi} \frac{g(\chi^{3})g(\overline{\chi^2})}{g(\chi)} \phi\chi \bigg(\frac{4\l}{27} \bigg) &= \phi \bigg(\frac{\l}{3} \bigg) \cdot  H_{p}\left[\begin{matrix} \frac{1}{3} & \frac{2}{3}\smallskip \\  &\frac{1}{2} \end{matrix} \; ; \; \l \right]\\
       &= \phi \bigg(\frac{\l}{3} \bigg) (N_{f}(\l,p) - 1).
\end{split}
\end{equation*}
The values at $\l=1$ follow from substitution and computing the discriminant of the polynomial $f_{\l}(x)$.
\end{proof}

\begin{proof}[The Proof of \Cref{thm:alg2}]

 We only show formula $(1)$ as the other two formulas follow similarly. \Cref{prop:gaussprops} and \Cref{prop:jacobiprops} give
\begin{equation*}
\begin{split}
 H_{p}\left[\begin{matrix} \frac{1}{2} & \frac{1}{4} & \frac{3}{4}\smallskip \\  &\frac{1}{3}&\frac{2}{3} \end{matrix} \; ; \; \l \right] 
 &= \frac{1}{p-1} \sum_{\chi} \frac{g(\chi^4)g(\overline{\chi^3})}{g(\chi)} \chi \bigg(-\frac{27\l}{256} \bigg)\\
 &= \frac{1}{p-1} \sum_{\chi} \frac{g(\chi^2)g(\phi \chi^{2})g(\overline{\chi^3})}{g(\phi)g(\chi)} \chi \bigg(-\frac{27\l}{16} \bigg)\\
 &= \frac{\phi(-1)}{p(p-1)} \sum_{\chi} g(\phi \chi) g(\phi \chi^2) g(\overline{\chi^3}) \chi \bigg(-\frac{27\l}{4} \bigg)\\
 &= \frac{\phi(-1)}{p(p-1)} \sum_{\chi} g(\phi \chi) g(\phi \chi^2) \bigg[\frac{\chi(-1)p}{g(\chi^3)} + (p-1) \delta(\chi^3) \bigg] \chi \bigg(- \frac{27\l}{4} \bigg)\\
 &= \frac{\phi(-1)}{p-1} \sum_{\chi} J(\phi \chi, \phi \chi^{2}) \chi \bigg( \frac{27\l}{4} \bigg)\\
 &+ \sum_{\chi} \bigg[-1 + \frac{\phi(-1)}{p} g(\phi \chi)g(\phi \chi^{2})\chi(-1)\bigg] \chi \bigg(\frac{27\l}{4} \bigg) \delta(\chi^3).\\
 \end{split}
 \end{equation*}
 Now letting $\chi \mapsto \phi \chi$ and using formula $(2)$ of \Cref{prop:jacobiprops} gives
 \begin{equation*}
 \begin{split}
 &\frac{\phi(-27 \l)}{p-1} \sum_{\chi} J(\phi \chi^{2},\chi) \chi \bigg(\frac{27\l}{4} \bigg)\\
 &+ \phi(27 \l) \sum_{\chi} \bigg[-1 + \frac{1}{p} g(\chi)g(\phi \chi^2)\chi(-1) \bigg] \chi \bigg(\frac{27\l}{4} \bigg) \delta(\phi \chi^{3})\\
 &= \phi(-3 \l) \sum_{a \in A_{\l,p}} \phi(a)
 \\
 &+ \phi(3 \l) \sum_{\chi \neq \phi} \bigg[-1 + \frac{1}{p} g(\chi)g(\phi \chi^2)\chi(-1) \bigg] \chi \bigg(\frac{27\l}{4} \bigg) \delta(\phi \chi^{3})
\end{split}
\end{equation*}
where $A_{\l,p} = \{a \in \fp \setminus \{0,1\} \, | \, 27\l(a^{3}-a^{2}) + 4 = 0 \}$ for a fixed $\l \in \fp^{\times}$.

Observe that the delta term is only nontrivial when $\chi = \eta_6$ or $\overline{\eta_6}$, where $\eta_{6}$ is an order six character of $\fphat$ that sends the generator to the primitive sixth root of unity $\zeta_{6}$. Therefore, if $p \equiv 1 \,  (\textrm{mod} \, \, 6)$ the delta term is 
$$\phi(3 \l) \bigg(-1 + \frac{1}{p}g(\eta_{6})g(\overline{\eta_{6}})\eta_{6}(-1) \bigg) \bigg[\eta_{6}\bigg(\frac{\l}{108} \bigg) + \overline{\eta_6}\bigg(\frac{\l}{108}\bigg) \bigg] = 0$$
by formula $(2)$ of \Cref{prop:gaussprops}.
\end{proof}

\section{Additional Cases and Numerical Conjectures}\label{sec:conj}

Hypergeometric data of lengths two and three that are defined over $\Q$ and primitive are studied carefully in this section. In particular, we determine all cases that are defined over $\Q$ and primitive when the data has a length of two or three. Then numerical conjectures and some new results are presented.

\subsection{Length Two Data}

In this section, all data will be considered $\textrm{modulo} \, \, \Z$. Let $\p(n)$ denote the Euler totient function. We consider all possible cases of length two hypergeometric data that are defined over $\Q$ and primitive. The $\alpha$ set must have the form $\alpha = \{ \frac{1}{s}, 1 - \frac{1}{s}\}$, where $s=2$ or $\p(s) = 2$. Therefore, the only eligible values of $s$ in the $\alpha$ set are $s=2,3,4,6$.

Now the $\beta$ set has the form $\beta = \{1, \frac{1}{t}\}$, where $\p(t) = 1$. The only two possibilities for $t$ are $t=1$ and $t=2$. Let $\text{HD}(s,t):= \{\{\frac{1}{s}, 1 - \frac{1}{s}\}, \{1, \frac{1}{t}\}\}$. The seven possible cases of length two data that are both defined over $\Q$ and primitive are $\text{HD}(2,1), \text{HD}(3,1), \text{HD}(4,1), \text{HD}(6,1)$, $\text{HD}(3,2)$, $\text{HD}(4,2),$ and $\text{HD}(6,2)$. The moments for the case $\text{HD}(2,1)$ are established by Ono, Saad, and Saikia in \eqref{eq:OSS} and the moments for the cases $\text{HD}(3,1)$ and $\text{HD}(4,1)$ are given in \Cref{thm:main3} and \Cref{thm:main4}, respectively. The remaining case with $t=1$, $\text{HD}(6,1)$, has the same moments, numerically, as the cases $\text{HD}(2,1), \text{HD}(3,1)$ and $\text{HD}(4,1)$.

\begin{conjecture}\label{conj:2F16}
If $m$ is a fixed positive integer then 
$$\lim_{p \to \infty} \frac{1}{p^{(\frac{m}{2}+1)}} \sum_{\l \in \mathbb{F}_{p}} H_{p}\left[\begin{matrix} \frac{1}{6} & \frac{5}{6} \smallskip \\  &1\end{matrix} \; ; \; \l \right]^{m}= \begin{cases}
0 &\quad \text{if m  is odd}\\
C(n) &\quad \text{if m = 2n is even}
\end{cases}$$

\end{conjecture}

Note that taking $d = 6$ in \Cref{cor:oddmoments} gives
$$\sum_{\l \in \fp} H_{p}\left[\begin{matrix} \frac{1}{6} & \frac{5}{6} \smallskip \\  &1\end{matrix} \; ; \; \l \right]^{2n+1} = 0$$ for primes $p \equiv 7 \pmod{12}$ and $p \equiv 11 \pmod{12}$ at nonnegative integers $n$. The elliptic curves associated with each non-algebraic case here are listed in \Cref{tab:Etfamily}.
\\

We now provide further evidence for \Cref{conj:2F16} by establishing the first and second moments for the $\text{HD}(6,1)$ case. The first moment follows from the orthogonality of characters and is recorded below.
\begin{proposition}\label{prop:firstmoment}
    Let $\{\alpha, \beta, \l\}$ be a hypergeometric datum which is both defined over $\Q$ and primitive. Define $H_{p}(\alpha, \beta, 0) = 1$. Then

    $$\sum_{\l \in \fp} H_{p}(\alpha, \beta, \l) = 0.$$
\end{proposition}

 The orthogonality of characters also allows us to prove the second moment for \Cref{conj:2F16}.

\begin{proposition}\label{prop:2F1secondmoment}
Let $p > 3$ be a prime. Then

$$\sum_{\l \in \fp^{\times}} H_{p}\left[\begin{matrix} \frac{1}{6} & \frac{5}{6} \smallskip \\  &1\end{matrix} \; ; \; \l \right]^{2} = \begin{cases}
p^{2}-3p-1 &\quad \text{if \, $p \equiv 1 \pmod{3}$ }\\
p^{2}-p-1 &\quad \text{if \, $p \equiv 2 \pmod{3}$}
\end{cases}$$
so
$$\lim_{p \to \infty} \frac{1}{p^{2}} \sum_{\l \in \mathbb{F}_{p}} H_{p
}\left[\begin{matrix} \frac{1}{6} & \frac{5}{6} \smallskip \\  &1\end{matrix} \; ; \; \l \right]^{2} = 1.$$
\end{proposition}

\begin{proof}
Observe that \Cref{prop:gaussprops} gives

$$H_{p}\left[\begin{matrix} \frac{1}{6} & \frac{5}{6} \smallskip \\  &1\end{matrix} \; ; \; \l \right] = \frac{1}{1-p} \sum_{\chi} \frac{g(\chi^{6})g(\chi)g(\overline{\chi})^{2}}{g(\chi^{2})g(\chi^{3})} \chi \bigg(\frac{\l}{432} \bigg).$$

Then 

\begin{equation*}
\begin{split}
    &\sum_{\l \in \fp^{\times}} H_{p}\left[\begin{matrix} \frac{1}{6} & \frac{5}{6} \smallskip \\  &1\end{matrix} \; ; \; \l \right]^{2}\\
    &= \frac{1}{(1-p)^2} \sum_{\chi} \frac{g(\chi^{6})g(\chi)g(\overline{\chi})^{2}}{g(\chi^{2})g(\chi^{3})} \sum_{\psi} \frac{g(\psi^{6})g(\psi)g(\overline{\psi})^{2}}{g(\psi^{2})g(\psi^{3})} \bigg( \sum_{\l \in \fq^{\times}} \chi \psi(\l) \bigg).
\end{split}
\end{equation*}

Now note that

$$\sum_{\l \in \fp^{\times}} \chi \psi(\l) = \begin{cases}
p -1 &\quad \text{if} \quad \text{$\psi = \overline{\chi}$}\\
0 &\quad \text{if} \quad \text{$\psi \neq \overline{\chi}$}.
\end{cases}$$

Therefore,

\begin{equation*}
\begin{split}
    \sum_{\l \in \fp^{\times}} H_{p}\left[\begin{matrix} \frac{1}{6} & \frac{5}{6} \smallskip \\  &1\end{matrix} \; ; \; \l \right]^{2} &= \frac{1}{p-1} \sum_{\chi} \frac{g(\chi^{6})g(\chi)g(\overline{\chi})^{2}}{g(\chi^{2})g(\chi^{3})}  \frac{g(\overline{\chi^{6}})g(\chi)g(\overline{\chi})^{2}}{g(\overline{\chi^{2}})g(\overline{\chi^{3}})}\\
    &= \frac{1}{p-1} \bigg(p^{2}+1 + \sum_{\chi \neq \eps, \phi} \frac{g(\chi^{6})g(\chi)g(\overline{\chi})^{2}}{g(\chi^{2})g(\chi^{3})}  \frac{g(\overline{\chi^{6}})g(\chi)g(\overline{\chi})^{2}}{g(\overline{\chi^{2}})g(\overline{\chi^{3}})} \bigg).
\end{split}
\end{equation*}

There are now two cases when $p \equiv 1 \pmod{3}$ and $p \equiv 2 \pmod{3}$. If $p \equiv 1 \pmod{3}$ then 

\begin{equation*}
\begin{split}
    \sum_{\l \in \fp^{\times}} H_{p}\left[\begin{matrix} \frac{1}{6} & \frac{5}{6} \smallskip \\  &1\end{matrix} \; ; \; \l \right]^{2} &= \frac{p^{3}-4p^{2}+2p+1}{p-1}\\
    &= p^{2}-3p-1,
\end{split}
\end{equation*}

by \Cref{prop:gaussprops}. If $p \equiv 2 \pmod{3}$ then a similar computation gives
\begin{equation*}
\begin{split}
    \sum_{\l \in \fp^{\times}} H_{p}\left[\begin{matrix} \frac{1}{6} & \frac{5}{6} \smallskip \\  &1\end{matrix} \; ; \; \l \right]^{2} &= \frac{p^{3}-2p^{2}+1}{p-1}\\
    &= p^{2}-3p-1.
\end{split}
\end{equation*}

\end{proof}

The remaining three cases of length two data, the $t=2$ cases, correspond to algebraic hypergeometric functions, discussed in Chapter 3 of \cite{long}. The algebraic formulas for these cases are given in \Cref{thm:alg1}. The even moments for these cases $\text{HD}(3,2), \text{HD}(4,2), \text{HD}(6,2)$ do not appear to match with a known sequence in OEIS. However, numerical data suggests the odd moments are zero in these three cases.

\subsection{Length Three Data}

We now give a similar analysis for the relevant cases of length three data. In particular, we find there are thirteen cases of hypergeometric data in length three which are defined over $\Q$ and primitive. It is well known that Euler's totient function $\p(n)$ is even for $n > 2$ so the $\alpha$ set must have the following form $\alpha = \{\frac{1}{2}, \frac{1}{s}, 1 - \frac{1}{s}\}$ for $s = 2,3,4$ and $6$, following the discussion for the length two data. The possibilities for the $\beta$ set are a subset of the possibilities for the $\alpha$ set in the case of length two data. The cases in which $\beta = \{1, \frac{1}{2}, \frac{1}{2}\}$ and $\beta = \{1, \frac{1}{s}, 1 - \frac{1}{s}\}$ are excluded to ensure the hypergeometric datum is primitive.

Let $\text{HD}(2,s,t):= \{ \{\frac{1}{2}, \frac{1}{s}, 1 - \frac{1}{s}\}, \{1, \frac{1}{t}, 1 - \frac{1}{t}\}\}$, where all values are considered $\text{mod} \, \, \Z$. The thirteen possible cases of length three data that are both defined over $\Q$ and primitive are $\text{HD}(2, 2, 1)$, $\text{HD}(2,3,1)$, $\text{HD}(2,4,1)$, $ \text{HD}(2,6,1)$, $\text{HD}(2, 2, 3)$, $ \text{HD}(2,2,4)$, $\text{HD}(2,2,6)$, $\text{HD}(2,3,4)$, $\text{HD}(2,3,6)$, $\text{HD}(2,4,3)$, $\text{HD}(2,4,6)$, $\text{HD}(2,6,3)$, and $\text{HD}(2,6,4)$. The three cases $\text{HD}(2,4,3)$, $\text{HD}(2,6,3)$, and $\text{HD}(2,6,4)$ are algebraic. The corresponding algebraic formulas are in \Cref{thm:alg2}.

All ten remaining cases are not algebraic. The moments for the first case, $\text{HD}(2,2,1)$, were determined by Ono, Saad, and Saikia in \cite{oss}. We observe the remaining nine non-algebraic cases have the same moments as the $\text{HD}(2,2,1)$ case, numerically. 

\begin{conjecture}\label{conj:3F2} If $m$ is a fixed positive integer then

 $$\lim_{p \to + \infty} \frac{1}{p^{{m+1}}} \sum_{\l \in \mathbb{F}_{p}} H_{p}\left[\begin{matrix} \frac{1}{2} & \frac{1}{s} & 1 - \frac{1}{s} \smallskip \\  &\frac{1}{t}&1-\frac{1}{t}\end{matrix} \; ; \; \l \right]^{m}$$
 
$$= \begin{cases} 
      0 &\quad \text{if m is odd}\\
      \displaystyle \sum_{i=0}^{m} (-1)^{i} \binom{m}{i} C(i) &\quad \text{if m is even} \\
   \end{cases}
$$

 for $(s,t) \in \{(3,1), (4,1), (6,1)\}$ and

 $$\lim_{p \to + \infty} \frac{1}{p} \sum_{\l \in \mathbb{F}_{p}} H_{p}\left[\begin{matrix} \frac{1}{2} & \frac{1}{s} & 1 - \frac{1}{s} \smallskip \\  &\frac{1}{t}&1-\frac{1}{t}\end{matrix} \; ; \; \l \right]^{m}$$
 
$$= \begin{cases} 
      0 \hspace{50mm} \text{if m is odd}\\
      \displaystyle \sum_{i=0}^{m} (-1)^{i} \binom{m}{i} C(i) \hspace{18mm} \text{if m is even} \\
   \end{cases}
$$

 for $(s,t) \in \{(2,3), (2,4), (2,6), (3,4), (3,6), (4,6)\}$.
\end{conjecture}

The discrepancy in the normalizing factors of \Cref{conj:3F2} is due to the zig-zag diagram $Z(HD)$ associated with each hypergeometric datum. See \cite{hgm} for background on zig-zag diagrams for hypergeometric data that is defined over $\Q$. If $w$ is the weight of $Z(HD)$ and $\text{min(Z)}$ is the associated minimum value then work of Katz \cite{Katzexp} shows 
$$H_{p}(\text{HD};\l) = O\left(p^{\frac{w-1}{2}+\text{min(Z)}}\right)$$ for any hypergeometric datum $\text{HD}$. Now observe that \text{min(Z) = 0} for the first three cases of \Cref{conj:3F2}, when $(s,t) \in \left\{(3,1),(4,1),(6,1)\right\}$. However, \text{min(Z) = $-1$} for the six remaining cases in \Cref{conj:3F2}.

Now the first moment for \Cref{conj:3F2} is zero by \Cref{prop:firstmoment}. We provide additional evidence for \Cref{conj:3F2} by showing the second moment in the three $t=1$ cases.

\begin{proposition}
Let $p > 3$ be a prime and fix $d \in \{3,4,6\}$. Then

$$\sum_{\l \in \fp^{\times}} H_{p}\left[\begin{matrix} \frac{1}{2} & \frac{1}{d} & 1 - \frac{1}{d} \smallskip \\  &1&1\end{matrix} \; ; \; \l \right]^{2} = 
\begin{cases} 
      p^{3}-4p^{2}-p-1 &\quad \text{if \, $p \equiv 1 \pmod{d}$}\\
      p^{3}-2p^{2}-p-1 &\quad \text{if \, $p \equiv -1 \pmod{d}$} \\
   \end{cases}$$
Therefore,

$$\lim_{p \to \infty} \frac{1}{p^{3}} \sum_{\l \in \mathbb{F}_{p}} H_{p}\left[\begin{matrix} \frac{1}{2} & \frac{1}{d} & 1 - \frac{1}{d} \smallskip \\  &1&1\end{matrix} \; ; \; \l \right]^{2} = 1.$$
\end{proposition}

\begin{proof}
The proof follows in the same way as the proof of \Cref{prop:2F1secondmoment}. 
\end{proof}
A similar argument can be used to establish the second moment for the six remaining cases of \Cref{conj:3F2}. The values of the even moments for length three data have a geometric meaning as the even traces of the matrices in the real orthogonal group $O_{3}$. That is, for $m$ even
$$\int_{O_{3}} (\Tr X)^{m} dX = \sum_{i=0}^{m} (-1)^{i} \binom{m}{i} C(i),$$
where the integral is with respect to the Haar measure on $O_{3}$.

The corresponding distribution was shown to be the so-called Batman distribution in \cite{oss}. This Batman distribution was made more explicit in the work of Saad \cite{saad}. Recently, Pujahari and Saikia have discovered the $O_{3}$ distribution appears for a few cases of $p$-adic hypergeometric functions \cite{padic2}. Further, Huang, Ono, and Saad \cite{matrix} showed the $O_{3}$ distribution also appears when counting matrix points on certain K3 surfaces. The hypergeometric differential equations for these length three defined over $\Q$ cases also appear in string theory \cite{HGk3} when studying fibrations of $K3$ surfaces. 

\bibliographystyle{plain}
\bibliography{ref2}

\end{document}